\documentclass[11pt]{amsart}
\usepackage{amsmath,amssymb,amsfonts,latexsym,graphicx,amsthm,mathdots}
\usepackage{hyperref}
\theoremstyle{plain}
\newtheorem{theorem}{Theorem}[section]
\newtheorem{lemma}[theorem]{Lemma}
\newtheorem{proposition}[theorem]{Proposition}

\theoremstyle{definition}
\newtheorem{conjecture}[theorem]{Conjecture}
\newtheorem{notation}[theorem]{Notation}

\newtheorem{question}[theorem]{Question}
\numberwithin{equation}{section}

\begin{document}

\title{Cyclotomic difference sets in finite fields}

\thanks{MSC2010: 05B10, 05B25, 11T22, 11T24, 65H10.
Keywords: difference set, Gauss sum, Jacobi sum, finite projective plane, discrete Fourier transform.
This work was partially supported by NSFC grant (11501011).}

\author{Binzhou Xia}
\address{School of Mathematics and Statistics\\University of Western Australia\\ Crawley 6009, WA\\ Australia}
\email{binzhouxia@pku.edu.cn}
\maketitle

\begin{abstract}
The classical problem of whether $m$th-powers with or without zero in a finite field $\mathbb{F}_q$ form a difference set has been extensively studied, and is related to many topics, such as flag transitive finite projective planes. In this paper new necessary and sufficient conditions are established including those via a system of polynomial equations on Gauss sums. The author thereby solves the problem for even $q$ which is neglected in the literature, and extends the nonexistence list for even $m$ up to $22$. Moreover, conjectures toward the complete classification are posed.
\end{abstract}

\section{Introduction}

A subset $D=\{a_1,\ldots,a_k\}$ in a group $F$ of order $v$ is said to be a \emph{$(v,k,\lambda)$-difference set} or simply a \emph{difference set} if for each nonidentity $a\in F$ there are exactly $\lambda$ ordered pairs $(a_s,a_t)\in D\times D$ such that $a_sa_t^{-1}=a$. Given a $(v,k,\lambda)$-difference set, we obtain instantly by simple counting that
\begin{equation}\label{11}
k(k-1)=\lambda(v-1).
\end{equation}
It is straightforward to check that any subset of a group $F$ with size $0$, $1$, $|F|-1$ or $|F|$ is a $(|F|,k,\lambda)$-difference set with $k-\lambda=0$ or $1$. Conversely, if $k-\lambda=0$ or $1$, then (\ref{11}) implies $k=0$, $1$, $v-1$ or $v$. For a $(v,k,\lambda)$-difference set $D$, we call the nonnegative integer $n:=k-\lambda$ the \emph{order} of $D$, and say $D$ is \emph{trivial} if $n\leqslant1$ and \emph{nontrivial} if $n>1$.

For comprehensive surveys on difference sets, the reader is referred to \cite[18 Part VI]{colbourn2010handbook} and \cite{jungnickel1992difference}. In this paper we focus on the case when $F$ is the additive group of a finite field and the nonzero elements in $D$ form a multiplicative subgroup of $F\setminus\{0\}$.

\begin{notation}
Let $q=mf+1$ be a power of a prime number $p$ with $m,f\in\mathbb{Z}_{>0}$. Denote the set of nonzero $m$th-powers in $\mathbb{F}_q$ by $H_{q,m}$ and $M_{q,m}:=H_{q,m}\cup\{0\}$.
\end{notation}

If $H_{q,m}$ is a $(q,f,\lambda)$-difference set in $\mathbb{F}_q^+$, then it is called an \emph{$m$th-cyclotomic difference set} or \emph{$m$th-power residue difference set}. If $M_{q,m}$ is a $(q,f+1,\lambda)$-difference set in $\mathbb{F}_q^+$, then it is called a \emph{modified $m$th-cyclotomic difference set} or \emph{modified $m$th-power residue difference set}. When not specifying the parameters, we will simply call them \emph{cyclotomic difference set} or \emph{modified cyclotomic difference set}, respectively. In view of (\ref{11}), a necessary condition for the cyclotomic $(q,f,\lambda)$-difference set is
\begin{equation}\label{12}
f-1=\lambda m,
\end{equation}
while a necessary condition for the modified cyclotomic $(q,f+1,\lambda)$-difference set is
\begin{equation}\label{13}
f+1=\lambda m.
\end{equation}

Research on cyclotomic and modified cyclotomic difference sets dates back to Paley in 1930s' \cite{paley1933orthogonal} when he used the quadratic residues in a finite filed to construct Hadamard matrices. Essentially, he proved that $H_{q,2}$ is a $(q,f,(q-3)/4)$-difference set under the condition $q\equiv3\pmod{4}$, which is no further restriction than (\ref{12}). Based on this result, it is clear to see that $M_{q,2}$ is a $(q,f+1,(q+1)/4)$-difference set under the condition $q\equiv3\pmod{4}$, which is no further restriction than (\ref{13}). About ten years after Paley's construction, Chowla \cite{chowla1944property} discovered a family of nontrivial quartic cyclotomic difference sets in $\mathbb{F}_p$ by showing that $H_{p,4}$ is a difference set when $p=1+4t^2$ for some odd integer $t$.

In 1953, Lehmer published a paper \cite{lehmer1953residue} investigating cyclotomic and modified cyclotomic difference sets in $\mathbb{F}_p$, where she established necessary and sufficient conditions for their existence via cyclotomic numbers and applied these to get fruitful results (see \cite[Chapter 2]{berndt1998gauss} for an introduction to cyclotomic numbers). She proved that neither $H_{p,m}$ nor $M_{p,m}$ is a nontrivial difference set in $\mathbb{F}_p$ with odd $m$, and determined all the nontrivial $m$th-cyclotomic and modified $m$th-cyclotomic difference sets in $\mathbb{F}_p$ for $4\leqslant m\leqslant8$: they are
\begin{align}
&\text{$H_{p,4}$ with $p=1+4t^2$ for some odd integer $t$,}\label{32}\\
&\text{$H_{p,8}$ with $p=1+8u^2=9+64v^2$ for some integers $u$ and $v$,}\label{33}\\
&\text{$M_{p,4}$ with $p=9+4t^2$ for some odd integer $t$,}\label{34}\\
&\text{$M_{p,8}$ with $p=49+8u^2=441+64v^2$ for some integers $u$ and $v$.}\label{35}
\end{align}
Note that the Pell equation $u^2-8v^2=1$ coming from (\ref{33}) forces $u$ and $v$ to be odd in order that $1+8u^2$ is prime (\cite[page 429]{lehmer1953residue}), while $u$ is odd and $v$ is even in (\ref{35}) for a similar reason (\cite[page 432]{lehmer1953residue}). On the other hand, it is not known yet whether there exist infinitely many primes $p$ as in (\ref{33}) or (\ref{35}), although Lehmer noticed that they are quite rare by computation results.

Since Lehmer's significant paper, cyclotomic numbers have been the main tool for studying existence of cyclotomic and modified cyclotomic difference sets. The criterion in terms of cyclotomic numbers established by Lehmer originally in $\mathbb{F}_p$ extends to general finite fields $\mathbb{F}_q$ in the same form, and then it is shown that, if $q$ is odd, the general finite field case does not give more examples than (\ref{32})--(\ref{35}) for $m$ odd or $4\leqslant m\leqslant8$ \cite{hall1965characters,storer1967cyclotomy}. As for the case when $q=p$, more nonexistence results on $m$th-cyclotomic and modified cyclotomic difference sets have been proved. They are now known to be nonexistent for $m=10$ \cite{MR0113851}, $12$ \cite{whiteman1960cyclotomic}, $14$ \cite{muskat1966cyclotomic}, $16$ \cite{evans1980bioctic}, $18$ \cite{baumert1967cyclotomic} and $20$ \cite{evans1999nonexistence}. Although widely believed that for any larger even $m$ neither $H_{p,m}$ nor $M_{p,m}$ forms a difference set, it has only been proved for some values of $m$ under extra condition: $m\equiv6\pmod{8}$ with $4\in H_{p,m}$ \cite{muskat1966cyclotomic}, and $m=24$ with $2\in H_{p,3}$ or $3\in H_{p,4}$ \cite{evans1983twenty}. One of the difficulties for higher powers is to evaluate the cyclotomic numbers, and also the work is quite laborious when $m$ grows.

In this paper, we alternate the approach by considering the relations that the corresponding Gauss sums must satisfy, which leads to overdetermined systems of polynomial equations. This enables us to determine $m$th-cyclotomic and modified cyclotomic difference sets in $\mathbb{F}_q$ up to $m=22$ including even $q$, and gives insight for general $m$.

There are definite links between difference sets and other structures (see for example \cite{baumert1971cyclic}). A remarkable one is that to finite projective planes, which we will illustrate in Subsection~\ref{sec5}. Physical applications of difference sets can be found in the references listed in \cite{byard2011lam}.

The layout of this paper is as follows. After this introduction, we will have in Section~\ref{sec3} preparing results on multiplicative characters and Gauss and Jacobi sums as well as discrete Fourier transform techniques which will be utilized in subsequent sections. In Section~\ref{sec2}, we establish new necessary and sufficient conditions for the existence of cyclotomic and modified cyclotomic difference sets in finite fields via multiplicative characters, Jacobi sums and Gauss sums, respectively. Based on these, we obtain results on the existence problem of $m$th-cyclotomic and modified cyclotomic difference sets in Section~\ref{sec4}. It is proved that the only existing one for odd $m$ is the modified $3$rd-cyclotomic difference set in $\mathbb{F}_{16}$ (see Theorem \ref{t8}), while the existence for even $m$ implies restrictions on the solutions of certain system of polynomial equations. In the final section, we discuss the computation results for these systems, which yields the determination of existence up to $m=22$ (see Theorem \ref{t11}) and suggests conjectures toward the complete classification.

\section{Preliminaries}\label{sec3}

First of all, we set up some notation.

\begin{notation}
Let $\mathbb{S}^1=\{z\in\mathbb{C}\mid|z|=1\}$ be the unit cycle on the complex plane. Denote $\zeta_r=e^{2\pi i/r}$ for any $r\in\mathbb{Z}_{>0}$. For a field $F$, let $F^*=F\setminus\{0\}$ be the set of nonzero elements in $F$, which constitutes a group under multiplication of the field.
\end{notation}

\subsection{Multiplicative characters}

Let $\chi$ be a character of the multiplicative group $\mathbb{F}_q^*$, i.e., a group homomorphism from $\mathbb{F}_q^*$ to $\mathbb{C}^*$. For any $s\in\mathbb{Z}$, the map $\chi^s$ from $\mathbb{F}_q^*$ to $\mathbb{C}^*$ defined by $\chi^s(\alpha)=(\chi(\alpha))^s$ for $\alpha\in\mathbb{F}_q^*$ is also a character of $\mathbb{F}_q^*$. Extend the domain of $\chi$ to $\mathbb{F}_q$ by setting
$$
\chi(0)=
\begin{cases}
1,\quad\text{if $\chi$ is trivial,}\\
0,\quad\text{if $\chi$ is nontrivial}
\end{cases}
$$
and call $\chi$ a \emph{multiplicative character} on $\mathbb{F}_q$. For any $s\in\mathbb{Z}$, the character $\chi^s$ of $\mathbb{F}_q^*$ is also extended to a multiplicative character on $\mathbb{F}_q$, and by $\chi^s(\alpha)$ we mean the image of $\alpha\in\mathbb{F}_q$ under $\chi^s$ rather than $(\chi(\alpha))^s$. Note that the equality $\chi^s(\alpha)=(\chi(\alpha))^s$ may not hold after $\chi^s$ is extended to a multiplicative character on $\mathbb{F}_q$. For example, if $\chi$ is nontrivial and $s$ is a positive integer such that $\chi^s$ is trivial, then $\chi^s(0)=1\neq0=(\chi(0))^s$.

Through this section, we will evaluate some multiplicative character sums, which turns out to play a central role in the subsequent section. For any nontrivial multiplicative character $\chi$ on $\mathbb{F}_q$, it is well known (see for example~\cite[Page~9]{berndt1998gauss}) that
\begin{equation}\label{15}
\sum\limits_{\alpha\in\mathbb{F}_q}\chi(\alpha)=0
\end{equation}

\begin{lemma}\label{l3}
Let $\chi$ be a multiplicative character of order $m$ on $\mathbb{F}_q$ and $s\in\mathbb{Z}$. Then
$$
\sum\limits_{\beta,\gamma\in H_{q,m}}\chi^s(\beta-\gamma)=f\sum\limits_{\alpha\in H_{q,m}}\chi^s(1-\alpha).
$$
\end{lemma}

\begin{proof}
\begin{align*}
\sum\limits_{\beta,\gamma\in H_{q,m}}\chi^s(\beta-\gamma)
=&\sum\limits_{\beta\in H_{q,m}}\sum\limits_{\gamma\in H_{q,m}}\chi^s(1-\beta^{-1}\gamma)\\
=&\sum\limits_{\beta\in H_{q,m}}\sum\limits_{\alpha\in H_{q,m}}\chi^s(1-\alpha)=f\sum\limits_{\alpha\in H_{q,m}}\chi^s(1-\alpha).\qedhere
\end{align*}
\end{proof}

\begin{notation}
For any multiplicative character $\chi$ of order $m$ on $\mathbb{F}_q$ and $\gamma\in\mathbb{F}_q^*$, let $A_{\chi,\gamma}=\{\alpha\in H_{q,m}\mid\chi(1-\alpha)=\chi(\gamma)\}$, $B_{q,m,\gamma}=\{(\alpha,\beta)\in H_{q,m}\times H_{q,m}\mid\alpha-\beta=\gamma\}$ and $C_{q,m,\gamma}=\{(\alpha,\beta)\in M_{q,m}\times M_{q,m}\mid\alpha-\beta=\gamma\}$.
\end{notation}

The following lemma is apparent.

\begin{lemma}\label{l4}
Let $\chi$ be a multiplicative character of order $m$ on $\mathbb{F}_q$ and $\gamma\in\mathbb{F}_q^*$. Then the following statements hold.
\begin{itemize}
\item[(a)] $(\alpha,\beta)\mapsto\alpha^{-1}\beta$ is a bijection from $B_{q,m,\gamma}$ to $A_{\chi,\gamma}$. In particular, $|A_{\chi,\gamma}|=|B_{q,m,\gamma}|$.
\item[(b)] $|C_{q,m,\gamma}|=|B_{q,m,\gamma}|+|C_{q,m,\gamma}\cap(H_{q,m}\times\{0\})|+|C_{q,m,\gamma}\cap(\{0\}\times H_{q,m})|$.
\end{itemize}
\end{lemma}

We express the multiplicative character sum in the next lemma in terms of $|A_{\chi,\gamma}|$.

\begin{lemma}\label{l2}
Let $\chi$ be a multiplicative character of order $m$ on $\mathbb{F}_q$ and $\gamma\in\mathbb{F}_q^*$. Then
$$
\sum\limits_{s=0}^{m-1}\chi^{-s}(\gamma)\sum\limits_{\alpha\in H_{q,m}}\chi^s(1-\alpha)=m|A_{\chi,\gamma}|+1.
$$
\end{lemma}

\begin{proof}
We have
\begin{align*}
&\sum\limits_{s=0}^{m-1}\chi^{-s}(\gamma)\sum\limits_{\alpha\in H_{q,m}}\chi^s(1-\alpha)\\
=&\sum\limits_{\alpha\in H_{q,m}}\sum\limits_{s=0}^{m-1}\chi^{-s}(\gamma)\chi^s(1-\alpha)\\
=&\sum\limits_{\alpha\in H_{q,m}\setminus\{1\}}\sum\limits_{s=0}^{m-1}\chi^{-s}(\gamma)\chi^s(1-\alpha)
+\sum\limits_{s=0}^{m-1}\chi^{-s}(\gamma)\chi^s(0)\\
=&\sum\limits_{\substack{\alpha\in H_{q,m}\setminus\{1\}\\\chi(1-\alpha)=\chi(\gamma)}}\sum\limits_{s=0}^{m-1}1
+\sum\limits_{\substack{\alpha\in H_{q,m}\setminus\{1\}\\\chi(1-\alpha)\neq\chi(\gamma)}}\sum\limits_{s=0}^{m-1}\left(\frac{\chi(1-\alpha)}{\chi(\gamma)}\right)^s
+\sum\limits_{s=0}^{m-1}\chi^{-s}(\gamma)\chi^s(0).
\end{align*}
For any $\alpha\in H_{q,m}\setminus\{1\}$ such that $\chi(1-\alpha)\neq\chi(\gamma)$,
$$
\sum\limits_{s=0}^{m-1}\left(\frac{\chi(1-\alpha)}{\chi(\gamma)}\right)^s
=\frac{\left(\frac{\chi(1-\alpha)}{\chi(\gamma)}\right)^m-1}{\frac{\chi(1-\alpha)}{\chi(\gamma)}-1}
=\frac{1-1}{\frac{\chi(1-\alpha)}{\chi(\gamma)}-1}=0.
$$
Hence it follows that
\begin{align*}
\sum\limits_{s=0}^{m-1}\chi^{-s}(\gamma)\sum\limits_{\alpha\in H_{q,m}}\chi^s(1-\alpha)=&\sum\limits_{\substack{\alpha\in H_{q,m}\setminus\{1\}\\\chi(1-\alpha)=\chi(\gamma)}}\sum\limits_{s=0}^{m-1}+\sum\limits_{s=0}^{m-1}\chi^{-s}(\gamma)\chi^s(0)\\
=&\sum\limits_{\substack{\alpha\in H_{q,m}\setminus\{1\}\\\chi(1-\alpha)=\chi(\gamma)}}m+\chi^0(\gamma)\chi^0(0)\\
=&\sum\limits_{\substack{\alpha\in H_{q,m}\\\chi(1-\alpha)=\chi(\gamma)}}m+1=m|A_{\chi,\gamma}|+1.\qedhere
\end{align*}
\end{proof}

\subsection{Gauss and Jacobi sums}

For a multiplicative character $\chi$ on $\mathbb{F}_q$, the \emph{Gauss sum} $G_q(\chi)$ is defined by
$$
G_q(\chi)=\sum\limits_{\alpha\in\mathbb{F}_q}\chi(\alpha)\zeta_p^{\mathrm{tr}(\alpha)},
$$
where $\mathrm{tr}$ is the trace map from $\mathbb{F}_q$ to $\mathbb{F}_p$. For multiplicative characters $\chi,\psi$ on $\mathbb{F}_q$, the \emph{Jacobi sum} $J_q(\chi,\psi)$ is defined by
$$
J_q(\chi,\psi)=\sum\limits_{\alpha\in\mathbb{F}_q}\chi(\alpha)\psi(1-\alpha).
$$
It is immediate from the definition that $J_q(\psi,\chi)=J_q(\chi,\psi)$.

We only list here some basic facts about Gauss and Jacobi sums which will be used in the sequel, and refer to~\cite[Chapters 1 and 2]{berndt1998gauss} for their proof and more properties of Gauss and Jacobi sums.

\begin{proposition}\label{p1}
Let $\chi$ be a multiplicative character of order $m$ on $\mathbb{F}_q$, and $s,t\in\mathbb{Z}$. Then the following statements hold.
\begin{itemize}
\item[(a)]
\begin{equation*}
|G_q(\chi^s)|=
\begin{cases}
\sqrt{q},\quad\text{if $s\not\equiv0\pmod{m}$},\\
0,\quad\text{if $s\equiv0\pmod{m}$}.
\end{cases}
\end{equation*}
\item[(b)] If $s\not\equiv0\pmod{m}$, then
\begin{equation*}
G_q(\chi^s)G_q(\chi^{-s})=\chi^s(-1)q.
\end{equation*}
\item[(c)] If $s\not\equiv0\pmod{m}$ or $t\not\equiv0\pmod{m}$, then
\begin{equation*}
J_q(\chi^s,\chi^t)=
\begin{cases}
G_q(\chi^s)G_q(\chi^t)/G_q(\chi^{s+t}),\quad\text{if $s+t\not\equiv0\pmod{m}$},\\
-\chi^s(-1),\quad\text{if $s+t\equiv0\pmod{m}$}.
\end{cases}
\end{equation*}
\item[(d)] If $m$ is even and $s\not\equiv0\pmod{m}$, then
\begin{equation*}
\chi^s(4)J_q(\chi^s,\chi^s)=J_q(\chi^s,\chi^{m/2}).
\end{equation*}
\end{itemize}
\end{proposition}

The following lemma evaluates sums of Jacobi sums with one character fixed.

\begin{lemma}\label{l1}
Let $\chi$ be a multiplicative character of order $m$ on $\mathbb{F}_q$. If $s\not\equiv0\pmod{m}$, then
$$
\sum\limits_{t=1}^{m-1}J_q(\chi^s,\chi^t)=1+m\sum\limits_{\alpha\in H_{q,m}}\chi^s(1-\alpha).
$$
\end{lemma}

\begin{proof}
Note that
\begin{align*}
\sum\limits_{t=1}^{m-1}J_q(\chi^s,\chi^t)
=&\sum\limits_{t=1}^{m-1}\sum\limits_{\beta\in\mathbb{F}_q}\chi^s(\beta)\chi^t(1-\beta)\\
=&\sum\limits_{\beta\in\mathbb{F}_q}\chi^s(\beta)\sum\limits_{t=1}^{m-1}\chi^t(1-\beta)\\
=&\sum\limits_{1-\beta\in H_{q,m}}\chi^s(\beta)(m-1)
+\sum\limits_{1-\beta\in\mathbb{F}_q^*\setminus H_{q,m}}\chi^s(\beta)\sum\limits_{t=1}^{m-1}\chi^t(1-\beta)\\
=&\sum\limits_{1-\beta\in H_{q,m}}\chi^s(\beta)(m-1)
+\sum\limits_{1-\beta\in\mathbb{F}_q^*\setminus H_{q,m}}\chi^s(\beta)\sum\limits_{t=0}^{m-1}\chi^t(1-\beta)\\
&-\sum\limits_{1-\beta\in\mathbb{F}_q^*\setminus H_{q,m}}\chi^s(\beta)\chi^0(1-\beta).
\end{align*}
For any $\beta\in\mathbb{F}_q$ such that $1-\beta\in\mathbb{F}_q^*\setminus H_{q,m}$, as $\chi(1-\beta)\neq1$, we have
$$
\sum\limits_{t=0}^{m-1}\chi^t(1-\beta)=\frac{(\chi(1-\beta))^m-1}{\chi(1-\beta)-1}=\frac{1-1}{\chi(1-\beta)-1}=0.
$$
It follows that
\begin{align*}
\sum\limits_{t=1}^{m-1}J_q(\chi^s,\chi^t)=&\sum\limits_{1-\beta\in H_{q,m}}\chi^s(\beta)(m-1)
-\sum\limits_{1-\beta\in\mathbb{F}_q^*\setminus H_{q,m}}\chi^s(\beta)\chi^0(1-\beta)\\
=&(m-1)\sum\limits_{1-\beta\in H_{q,m}}\chi^s(\beta)-\sum\limits_{1-\beta\in\mathbb{F}_q^*\setminus H_{q,m}}\chi^s(\beta)\\
=&(m-1)\sum\limits_{\alpha\in H_{q,m}}\chi^s(1-\alpha)-\sum\limits_{\alpha\in\mathbb{F}_q^*\setminus H_{q,m}}\chi^s(1-\alpha)\\
=&m\sum\limits_{\alpha\in H_{q,m}}\chi^s(1-\alpha)-\sum\limits_{\alpha\in\mathbb{F}_q^*}\chi^s(1-\alpha)\\
=&m\sum\limits_{\alpha\in H_{q,m}}\chi^s(1-\alpha)-\sum\limits_{\alpha\in\mathbb{F}_q}\chi^s(1-\alpha)+\chi^s(1)\\
=&1+m\sum\limits_{\alpha\in H_{q,m}}\chi^s(1-\alpha).\qedhere
\end{align*}
\end{proof}

\subsection{Discrete Fourier transform}

For a complex-valued function $X$ on $\mathbb{Z}/r\mathbb{Z}$, the discrete Fourier transform (DFT) of $X$, denoted by $\hat{X}$, is the complex-valued function on $\mathbb{Z}/r\mathbb{Z}$ defined by
$$
\hat{X}(s)=\sum\limits_{t=0}^{r-1}\zeta_r^{-st}X(t),\quad s=0,\dots,r-1.
$$
Here are two basic formulae for DFT, the convolution formula and the inverse formula, see for example \cite[page 36]{terras1999fourier} for a proof.

\begin{proposition}\label{p3}
The following statements hold.
\begin{itemize}
\item[(a)] If $W$, $X$ and $Y$ are complex-valued functions on $\mathbb{Z}/r\mathbb{Z}$ with
$$
W(s)=\sum\limits_{t=0}^{r-1}X(t)Y(s-t),\quad s=0,\dots,r-1,
$$
then $\hat{W}(s)=\hat{X}(s)\hat{Y}(s)$ for all $s\in\mathbb{Z}/r\mathbb{Z}$.
\item[(b)] If $X$ is a complex-valued function on $\mathbb{Z}/r\mathbb{Z}$, then $X(s)=\hat{\hat{X}}(-s)/r$ for all $s\in\mathbb{Z}/r\mathbb{Z}$. In particular, DFT is an isomorphism on the complex vector space of complex-valued functions on $\mathbb{Z}/r\mathbb{Z}$.
\end{itemize}
\end{proposition}

We note that since complex-valued functions on $\mathbb{Z}/r\mathbb{Z}$ are determined by their values on the $r$ points $0,\dots,r-1$, Proposition \ref{p3} can be translated into the language of $r$-dimensional complex vectors. This will be more convenient to apply in cases.

\section{Necessary and sufficient conditions}\label{sec2}

\subsection{Cyclotomic difference sets}

Before we give the existence criterions for cyclotomic difference sets, recall the necessary condition (\ref{12}) for cyclotomic $(q,f,\lambda)$-difference sets.

\begin{theorem}\label{t9}
Suppose $f-1=\lambda m$ and let $\chi$ be a multiplicative character of order $m$ on $\mathbb{F}_q$. Then $H_{q,m}$ is a $(q,f,\lambda)$-difference set in $\mathbb{F}_q$ if and only if
\begin{equation}\label{26}
\sum\limits_{\alpha\in H_{q,m}}\chi^s(1-\alpha)=0,\quad s=1,\ldots,m-1.
\end{equation}
\end{theorem}

\begin{proof}
First suppose that $H_{q,m}$ is a $(q,f,\lambda)$-difference set. In view of (\ref{15}) we then have for $s=1,\ldots,m-1$ that
$$
\sum\limits_{\beta,\gamma\in H_{q,m}}\chi^s(\beta-\gamma)=\sum\limits_{\substack{\beta,\gamma\in H_{q,m}\\\beta\neq\gamma}}\chi^s(\beta-\gamma)
=\lambda\sum\limits_{\alpha\in\mathbb{F}_q^*}\chi^s(\alpha)=0.
$$
This leads to (\ref{26}) by Lemma \ref{l3}.

Next suppose that (\ref{26}) holds. Let $\gamma$ be an arbitrary element in $\mathbb{F}_q^*$. Then
$$
\sum\limits_{s=0}^{m-1}\chi^{-s}(\gamma)\sum\limits_{\alpha\in H_{q,m}}\chi^s(1-\alpha)=\sum\limits_{\alpha\in H_{q,m}}\chi^0(1-\alpha)=f.
$$
It follows that $|A_{\chi,\gamma}|=(f-1)/m=\lambda$ by Lemma \ref{l2}, and so $|B_{q,m,\gamma}|=\lambda$ according to Lemma \ref{l4}(a). By the definition of difference sets, this completes the proof.
\end{proof}

Combining Lemma \ref{l1} and Theorem \ref{t9} we obtain a necessary and sufficient condition of cyclotomic difference sets via Jacobi sums.

\begin{theorem}\label{t5}
Suppose $f-1=\lambda m$, and let $\chi$ be a multiplicative character of order $m$ on $\mathbb{F}_q$. Then $H_{q,m}$ is a $(q,f,\lambda)$-difference set in $\mathbb{F}_q$ if and only if
\begin{equation}\label{30}
\sum\limits_{t=1}^{m-1}J_q(\chi^s,\chi^t)=1,\quad s=1,\dots,m-1.
\end{equation}
\end{theorem}

In light of Proposition \ref{p1}(c), (\ref{30}) can be rewritten by Gauss sums.

\begin{theorem}\label{t2}
Suppose $f-1=\lambda m$, and let $\chi$ be a multiplicative character of order $m$ on $\mathbb{F}_q$. Then $H_{q,m}$ is a $(q,f,\lambda)$-difference set in $\mathbb{F}_q$ if and only if
\begin{equation}\label{16}
\sum\limits_{\substack{t=1\\ t\neq s}}^{m-1}\chi^t(-1)G_q(\chi^t)G_q(\chi^{s-t})=(1+\chi^s(-1))G_q(\chi^s),\quad s=1,\dots,m-1.
\end{equation}
\end{theorem}

\begin{proof}
Utilizing Proposition \ref{p1}, we reformulate (\ref{30}) to
$$
\sum\limits_{\substack{t=1\\ s+t\neq m}}^{m-1}\frac{G_q(\chi^s)G_q(\chi^t)}{G_q(\chi^{s+t})}-\chi^s(-1)=1,\quad s=1,\dots,m-1,
$$
which is equivalent to
$$
\sum\limits_{\substack{t=1\\ t\neq m-s}}^{m-1}\frac{\chi^t(-1)G_q(\chi^t)G_q(\chi^{m-s-t})}{G_q(\chi^{m-s})}=1+\chi^s(-1),\quad s=1,\dots,m-1.
$$
After multiplying both sides by $G_q(\chi^{m-s})$ and replacing $s$ by $m-s$, this turns out to be (\ref{16}). Hence the theorem follows by Theorem \ref{t5}.
\end{proof}

\subsection{Modified cyclotomic difference sets}

Parallel with cyclotomic difference sets we can establish existence criterions for modified cyclotomic difference sets. Recall the necessary condition (\ref{13}) for the modified cyclotomic $(q,f+1,\lambda)$-difference sets.

\begin{theorem}\label{t6}
Suppose $f+1=\lambda m$, and let $\chi$ be a multiplicative character of order $m$ on $\mathbb{F}_q$. Then $M_{q,m}$ is a $(q,f+1,\lambda)$-difference set in $\mathbb{F}_q$ if and only if
\begin{equation}\label{27}
\sum\limits_{\alpha\in H_{q,m}}\chi^s(1-\alpha)=-1-\chi^s(-1),\quad s=1,\ldots,m-1.
\end{equation}
\end{theorem}

\begin{proof}
First suppose that $M_{q,m}$ is a $(q,f+1,\lambda)$-difference set. Then for $s=1,\ldots,m-1$,
$$
\sum\limits_{\beta,\gamma\in M_{q,m}}\chi^s(\beta-\gamma)=\sum\limits_{\substack{\beta,\gamma\in M_{q,m}\\ \beta\neq\gamma}}\chi^s(\beta-\gamma)=\lambda\sum\limits_{\alpha\in\mathbb{F}_q^*}\chi^s(\alpha)=0.
$$
On the other hand,
\begin{align*}
\sum\limits_{\beta,\gamma\in M_{q,m}}\chi^s(\beta-\gamma)=&\sum\limits_{\beta,\gamma\in H_{q,m}}\chi^s(\beta-\gamma)+\sum\limits_{\beta\in H_{q,m}}\chi^s(\beta)+\sum\limits_{\gamma\in H_{q,m}}\chi^s(-\gamma)\\
=&\sum\limits_{\beta,\gamma\in H_{q,m}}\chi^s(\beta-\gamma)+f+f\chi^s(-1).
\end{align*}
Hence
\begin{equation*}
\sum\limits_{\beta,\gamma\in H_{q,m}}\chi^s(\beta-\gamma)=-f(1+\chi^s(-1)),\quad s=1,\ldots,m-1,
\end{equation*}
and thus we get (\ref{27}) by virtue of Lemma \ref{l3}.

Now suppose conversely that (\ref{27}) holds. Let $\gamma$ be an arbitrary element in $\mathbb{F}_q^*$. Then
$$
\sum\limits_{s=0}^{m-1}\chi^{-s}(\gamma)\sum\limits_{\alpha\in H_{q,m}}\chi^s(1-\alpha)=f-\sum\limits_{s=1}^{m-1}\chi^{-s}(\gamma)(1+\chi^s(-1)).
$$
We thereby deduce from Lemmas \ref{l4} and \ref{l2} that
\begin{equation}\label{14}
m|B_{q,m,\gamma}|+1=f-\sum\limits_{s=1}^{m-1}\chi^{-s}(\gamma)(1+\chi^s(-1))
\end{equation}
Recall that $C_{q,m,\gamma}=\{(\alpha,\beta)\in M_{q,m}\times M_{q,m}\mid\alpha-\beta=\gamma\}$. It suffices to show $|C_{q,m,\gamma}|=\lambda$ by the definition of difference set. Observe that $\chi(\gamma)\neq1$ implies $|C_{q,m,\gamma}\cap(H_{q,m}\times\{0\})|=0$ while $\chi(\gamma)\neq\chi(-1)$ implies $|C_{q,m,\gamma}\cap(\{0\}\times H_{q,m})|=0$. The cases for $\chi(\gamma)$ divide into the following four.

\underline{Case 1. $\chi(\gamma)\neq1$ or $\chi(-1)$.} In this case
$$
\sum\limits_{s=1}^{m-1}\chi^{-s}(\gamma)(1+\chi^s(-1))=-2,
$$
whence (\ref{14}) gives $|B_{q,m,\gamma}|=\lambda$. We then conclude $|C_{q,m,\gamma}|=|B_{q,m,\gamma}|=\lambda$ viewing Lemma \ref{l4}(b).

\underline{Case 2. $\chi(\gamma)=1\neq\chi(-1)$.} Then (\ref{14}) gives $m|B_{q,m,\gamma}|+1=f-(m-1-1)$, i.e., $|B_{q,m,\gamma}|=\lambda-1$. Moreover,   \begin{equation}\label{17}
C_{q,m,\gamma}\cap(H_{q,m}\times\{0\})=\{(\gamma,0)\},
\end{equation}
so $|C_{q,m,\gamma}|=|B_{q,m,\gamma}|+1=\lambda$ by Lemma \ref{l4}(b).

\underline{Case 3. $\chi(\gamma)=\chi(-1)\neq1$.} In this case, (\ref{14}) leads to $|B_{q,m,\gamma}|=\lambda-1$, and it follows by Lemma \ref{l4}(b) that $|C_{q,m,\gamma}|=|B_{q,m,\gamma}|+1=\lambda$ since
\begin{equation}\label{18}
C_{q,m,\gamma}\cap(\{0\}\times H_{q,m})=\{(0,-\gamma)\}.
\end{equation}

\underline{Case 4. $\chi(\gamma)=1=\chi(-1)$.} In this case, (\ref{14}) leads to $|B_{q,m,\gamma}|=\lambda-2$, and it follows by Lemma \ref{l4}(b) that $|C_{q,m,\gamma}|=|B_{q,m,\gamma}|+2=\lambda$ since we have both (\ref{17}) and (\ref{18}).
\end{proof}

Combination of Lemma \ref{l1} and Theorem \ref{t6} leads to a necessary and sufficient condition of modified cyclotomic difference sets via Jacobi sums.

\begin{theorem}\label{t7}
Suppose $f+1=\lambda m$, and let $\chi$ be a multiplicative character of order $m$ on $\mathbb{F}_q$. Then $M_{q,m}$ is a $(q,f+1,\lambda)$-difference set in $\mathbb{F}_q$ if and only if
\begin{equation}\label{31}
\sum\limits_{t=1}^{m-1}J_q(\chi^s,\chi^t)=1-m-m\chi^s(-1),\quad s=1,\dots,m-1.
\end{equation}
\end{theorem}

Along the same lines as Theorem \ref{t2}, we can reformulate (\ref{31}) by Gauss sums as follows.

\begin{theorem}\label{t3}
Suppose $f+1=\lambda m$, and let $\chi$ be a multiplicative character of order $m$ on $\mathbb{F}_q$. Then $M_{q,m}$ is a $(q,f+1,\lambda)$-difference set in $\mathbb{F}_q$ if and only if
\begin{equation}\label{20}
\sum\limits_{\substack{t=1\\ t\neq s}}^{m-1}\chi^t(-1)G_q(\chi^t)G_q(\chi^{s-t})=(1-m)(1+\chi^s(-1))G_q(\chi^s),\quad s=1,\dots,m-1.
\end{equation}
\end{theorem}

\section{Existence conditions via polynomial equations}\label{sec4}

\subsection{System on $g$-level}

We investigate the existence problem for cyclotomic and modified cyclotomic difference sets based on the criterions obtained in the previous section.

First we embark on the case when $m$ is odd. It is already known in this case that neither $H_{q,m}$ nor $M_{q,m}$ form a nontrivial difference set in $\mathbb{F}_q$ if $q$ is odd \cite[Chapter 1, Part 1]{storer1967cyclotomy}. However, the parity argument for cyclotomic numbers used to prove this fact does not appeal to even $q$'s. In Theorem \ref{t8} below, we deal with the case when $m$ is odd in a uniform way for both even and odd $q$'s. It turns out that the only existent one is $M_{16,3}$, which is not mentioned in the literature.

\begin{theorem}\label{t8}
Suppose that $m$ is odd. Then the following statements hold.
\begin{itemize}
\item[(a)] $H_{q,m}$ is never a nontrivial difference set in $\mathbb{F}_q$.
\item[(b)] $M_{q,m}$ is a nontrivial difference set in $\mathbb{F}_q$ if and only if $(q,m)=(16,3)$.
\end{itemize}
\end{theorem}

\begin{proof}
Suppose that $H_{q,m}$ or $M_{q,m}$ is a nontrivial difference set. As a consequence, both $m$ and $f$ are greater than $1$. Let $\chi$ be a multiplicative character of order $m$ on $\mathbb{F}_q$. Define a function $g$ on $\mathbb{Z}/m\mathbb{Z}$ by
$$
g(s)=G_q(\chi^s)/\sqrt{q},\quad s=1,\dots,m-1,
$$
and
\begin{equation*}
g(0)=
\begin{cases}
-1/\sqrt{q},\quad\text{if $H_{q,m}$ is a nontrivial difference set},\\
(m-1)/\sqrt{q},\quad\text{if $M_{q,m}$ is a nontrivial difference set}.
\end{cases}
\end{equation*}
Noticing $\chi(-1)=1$ as $m$ is odd, we have
\begin{equation}
\sum\limits_{t=0}^{m-1}g(t)g(s-t)=0,\quad s=1,\dots,m-1
\end{equation}
by Theorems \ref{t2} and \ref{t3}, and
\begin{equation}
g(s)g(-s)=1,\quad s=1,\dots,m-1
\end{equation}
by Proposition \ref{p1}(b). Define a complex-valued function $W$ on $\mathbb{Z}/m\mathbb{Z}$ by
$$
W(s)=\sum\limits_{t=0}^{m-1}g(t)g(s-t),\quad s\in\mathbb{Z}.
$$
It follows that $W(1)=\dots=W(m-1)=0$ and
$$
W(0)=g(0)^2+\sum\limits_{t=1}^{m-1}g(t)g(-t)=g(0)^2+m-1.
$$
Relying on Proposition \ref{p3} we have
$$
\hat{g}(s)^2=\hat{W}(s)=\sum\limits_{t=0}^{m-1}\zeta_m^{-st}W(t)=W(0)=g(0)^2+m-1,\quad s\in\mathbb{Z},
$$
\begin{equation}\label{24}
\sum\limits_{r=0}^{m-1}\zeta_m^{sr}\hat{g}(r)\sum\limits_{t=0}^{m-1}\zeta_m^{-st}\hat{g}(t)=\hat{\hat{g}}(-s)\hat{\hat{g}}(s)=m^2g(s)g(-s)=m^2,\quad s=1,\dots,m-1
\end{equation}
and
\begin{equation}\label{25}
\sum\limits_{s=0}^{m-1}\hat{g}(s)=\hat{\hat{g}}(0)=mg(0).
\end{equation}
Therefore, $\hat{g}(s)=\varepsilon(s)\sqrt{g(0)^2+m-1}$ with $\varepsilon(s)=\pm1$ for $s\in\mathbb{Z}$, and substituting this into (\ref{24}) and (\ref{25}) we get
\begin{equation}\label{28}
\sum\limits_{r=0}^{m-1}\zeta_m^{sr}\varepsilon(r)\sum\limits_{t=0}^{m-1}\zeta_m^{-st}\varepsilon(t)=\frac{m^2}{g(0)^2+m-1},\quad s=1,\dots,m-1
\end{equation}
and
\begin{equation}\label{29}
\sum\limits_{s=0}^{m-1}\varepsilon(s)=\frac{mg(0)}{\sqrt{g(0)^2+m-1}}.
\end{equation}
Note that $\sum_{s=0}^{m-1}\varepsilon(s)$ is an odd integer as $m$ is odd. We proceed according to the three cases below.

\underline{Case 1. $H_{q,m}$ is a nontrivial difference set.} We deduce from (\ref{29}) that
$$
\left(\frac{mg(0)}{\sqrt{g(0)^2+m-1}}\right)^2\geqslant1,
$$
i.e., $g(0)^2\geqslant1/(m+1)$. This yields $q\leqslant m+1$, which violates the condition $f>1$.

\underline{Case 2. $p>2$ and $M_{q,m}$ is a nontrivial difference set.} In this case,
$$
\frac{m^2(m-1)}{m-1+q}=\left(\frac{mg(0)}{\sqrt{g(0)^2+m-1}}\right)^2
$$
is an odd integer by (\ref{29}). However, this is a contradiction as $m-1$ is even and $m-1+q$ is odd.

\underline{Case 3. $p=2$ and $M_{q,m}$ is a nontrivial difference set.} Suppose $M_{q,m}$ is a $(q,f+1,\lambda)$-difference set. Then $q=mf+1=\lambda m^2-m+1$ by (\ref{13}). In view of (\ref{28}),
$$
\frac{q}{\lambda(m-1)}=\frac{m^2}{g(0)^2+m-1}
$$
is an algebraic integer, and thus an integer as it is rational. On the other hand, $(m-1)/\lambda=m^2(m-1)/(m-1+q)$ is an odd integer as shown in the previous case. We thereby conclude that $\lambda=m-1$ is a power of $2$, whence $q=\lambda m^2-m+1=(m-1)^2(m+1)$. This implies that $m+1$ is also a power of $2$, so we have $m=3$. Thus $\lambda=2$ and $q=16$.

Conversely, consider $\mathbb{F}_{16}$ as the splitting field of $x^4+x+1$ over $\mathbb{F}_2$. Let $\omega$ be a root of $x^4+x+1=0$ in $\mathbb{F}_{16}$ and $\chi$ be a multiplicative character on $\mathbb{F}_{16}$ such that $\chi(\omega)=\zeta_3$. It is easy to check that $1-\omega^3=\omega^{14}$, $1-\omega^6=\omega^{13}$, $1-\omega^9=\omega^7$ and $1-\omega^{12}=\omega^{11}$. Now for $s=1,2$
$$
\sum\limits_{\alpha\in H_{16,3}}\chi^s(1-\alpha)=\sum\limits_{t=0}^4\chi^s(1-\omega^{3t})=\zeta_3^{2s}+\zeta_3^s+\zeta_3^s+\zeta_3^{2s}=-2.
$$
Thus $M_{16,3}$ is a $(16,6,2)$-difference set by Theorem \ref{t6}.
\end{proof}

During the proof of Theorem \ref{t8}, it is the relations of Gauss sums, with no need to evaluate them, from (\ref{16}) and (\ref{20}) as well as Proposition \ref{p1} that rule out the possibility of cyclotomic and modified cyclotomic difference sets. This suggests us to approach the case when $m$ is even in the same vein, namely studying the relations that Gauss sums necessarily satisfy.

\begin{theorem}\label{t1}
Suppose $m$ is even. If $H_{q,m}$ or $M_{q,m}$ is a difference set in $\mathbb{F}_q$, then the system of equations
\begin{equation}\label{22}
\begin{cases}
\sum\limits_{t=0}^{2s}(-1)^tg_tg_{2s-t}+\sum\limits_{t=2s+1}^{m-1}(-1)^tg_tg_{m+2s-t}=0,\quad s=1,\dots,\frac{m}{2}-1,\\
g_sg_{m-s}=(-1)^s,\quad s=1,\dots,\frac{m}{2},\\
h^sg_sg_{\frac{m}{2}+s}=g_{2s}g_{\frac{m}{2}},\quad s=1,\dots,\frac{m}{2}-1,\\
h^\frac{m}{2}=1
\end{cases}
\end{equation}
in the unknowns $g_0,g_1,\dots,g_{m-1},h$ has a solution in $\mathbb{R}\times(\mathbb{S}^1)^m$ with $g_0=-1/\sqrt{q}$ or $(m-1)/\sqrt{q}$ respectively.
\end{theorem}

\begin{proof}
Let $\chi$ be a multiplicative character of order $m$ on $\mathbb{F}_q$,
$$
g_s=G_q(\chi^s)/\sqrt{q},\quad s=1,\dots,m-1,
$$
\begin{equation*}
g_0=
\begin{cases}
-1/\sqrt{q},\quad\text{if $H_{q,m}$ is a difference set},\\
(m-1)/\sqrt{q},\quad\text{if $M_{q,m}$ is a difference set},
\end{cases}
\end{equation*}
and $h=\chi(4)$. Since $m$ is even, we derive from (\ref{12}) and (\ref{13}) that $f$ is odd, and thus $\chi(-1)=-1$. It follows from Theorems \ref{t2} and \ref{t3} that
$$
\sum\limits_{t=0}^s(-1)^tg_tg_{s-t}+\sum\limits_{t=s+1}^{m-1}(-1)^tg_tg_{m+s-t}=0,\quad s=1,\dots,m-1.
$$
In particular, taking even $s$ gives the first line of (\ref{22}). By Proposition \ref{p1} we have $|g_1|=\dots=|g_{m-1}|=1$,
$$
g_sg_{m-s}=(-1)^s,\quad s=1,\dots,m-1
$$
and
\begin{equation}\label{19}
h^s\frac{g_sg_s}{g_{2s}}=\frac{g_sg_{\frac{m}{2}}}{g_{\frac{m}{2}+s}},\quad s=1,\dots,\frac{m}{2}-1.
\end{equation}
Hence the second line of (\ref{22}) holds, and (\ref{19}) implies the third line of (\ref{22}). Finally, $h=\chi^2(2)$ satisfies $h^{m/2}=1$ and $|h|=1$. This completes the proof.
\end{proof}

For an even $m$, we call (\ref{22}) the \emph{system of order $m$ on $g$-level}. Note that if we count the subscript of $g$ modulo $m$ in the system of order $m$ on $g$-level, then the first line of (\ref{22}) can be written more concisely as
$$
\sum\limits_{t=0}^{m-1}(-1)^tg_tg_{2s-t}=0,\quad s=1,\dots,\frac{m}{2}-1.
$$

\begin{notation}
Denote the affine variety consisting of solutions $(g_0,g_1,\dots,g_{m-1},h)\in\mathbb{C}^{m+1}$ to the system of order $m$ on $g$-level by $L_m$.
\end{notation}

Here are some observations on $L_m$.

\begin{proposition}\label{p4}
Suppose that $m$ is even.
\begin{itemize}
\item[(a)] If $(g_0,g_1,\dots,g_{m-1},h)\in L_m$, then $(-g_0,-g_1,\dots,-g_{m-1},h)\in L_m$.
\item[(b)] If $(g_0,g_1,\dots,g_{m-1},h)\in L_m$, then for any integer $r$,
$$
(g_0,\zeta_m^rg_1,\dots,\zeta_m^{(m-1)r}g_{m-1},h)\in L_m.
$$
\item[(c)] If $(g_0,g_1,\dots,g_{m-1},h)\in L_m$, then for any integer $r$ which is coprime to $m$, $(g_0,g_r,\dots,g_{(m-1)r},h)$ with subscripts modulo $m$ lies in $L_m$.
\item[(d)] There exists $(g_0,g_1,\dots,g_{m-1},h)\in L_m$ such that $g_0=m/2-1$; if $m+1$ is a prime power then there exists $(g_0,g_1,\dots,g_{m-1},h)\in L_m$ such that $g_0^2=1/(m+1)$. In particular, $L_m$ is nonempty.
\end{itemize}
\end{proposition}

\begin{proof}
Parts (a)--(c) are straightforward. We only need to prove (d). If $m\equiv0\pmod{4}$, then take $g_0=m/2-1$,
$$
g_s=(-1)^\frac{(s-1)(s-2)}{2}\zeta_\frac{m}{2}^s,\quad s=1,\dots,m-1
$$
and $h=-1$. If $m\equiv2\pmod{4}$, then take $g_0=(-1)^{(m+6)(m-6)/32}(m/2-1)$,
$$
g_s=(-1)^\frac{(4s+m+2)(4s+m-2)}{32}\zeta_{2m}^s,\quad s=1,\dots,m-1
$$
and $h=1$. One verifies directly that $(g_0,g_1,\dots,g_{m-1},h)$ is a solution of (\ref{22}) with $g_0=\pm(m/2-1)$. Thus by part (a) we conclude that (\ref{22}) has a solution with $g_0=m/2-1$. Now suppose that $m+1$ is a prime power. Then $\{1\}$ is an $m$th-cyclotomic $(m+1,1,0)$-difference set in $\mathbb{F}_{m+1}$. By Theorem \ref{t1}, (\ref{22}) has a solution $(g_0,g_1,\dots,g_{m-1},h)$ with $g_0^2=1/(m+1)$.
\end{proof}

\begin{theorem}\label{t10}
Suppose that $m$ is even.
\begin{itemize}
\item[(a)] If each $(g_0,g_1,\dots,g_{m-1},h)\in L_m\cap(\mathbb{R}^*\times(\mathbb{S}^1)^m)$ satisfies $g_0^2\geqslant1/(m+1)$, then $H_{q,m}$ is not a nontrivial difference set in $\mathbb{F}_q$.
\item[(b)] If each $(g_0,g_1,\dots,g_{m-1},h)\in L_m\cap(\mathbb{R}^*\times(\mathbb{S}^1)^m)$ satisfies either $g_0^2\geqslant1$ or $g_0^2=1/(m+1)$, then neither $H_{q,m}$ nor $M_{q,m}$ is a nontrivial difference set in $\mathbb{F}_q$.
\end{itemize}
\end{theorem}

\begin{proof}
First we prove part (a). Suppose that $H_{q,m}$ is a nontrivial difference set in $\mathbb{F}_q$. Then as Theorem \ref{t1} asserts, (\ref{22}) has a solution $(g_0,g_1,\dots,g_{m-1},h)\in\mathbb{R}\times(\mathbb{S}^1)^m$ such that $g_0=-1/\sqrt{q}$. However, $g_0^2\geqslant1/(m+1)$, whence $q=m+1$. It follows that $f=1$ and thus $H_{q,m}$ is a trivial difference set. This contradiction shows that (a) is true.

Now we turn to the proof for part (b). Under the assumption in (b), since we can deduce that $g_0^2\geqslant1/(m+1)$, $H_{q,m}$ is not a nontrivial difference set in $\mathbb{F}_q$. Suppose that $M_{q,m}$ is a nontrivial $(q,f+1,\lambda)$-difference set in $\mathbb{F}_q$. By Theorem \ref{t1}, (\ref{22}) has a solution $(g_0,g_1,\dots,g_{m-1},h)\in\mathbb{R}\times(\mathbb{S}^1)^m$ such that $g_0=(m-1)/\sqrt{q}$. Viewing (\ref{13}), we then have
$$
g_0^2=\frac{(m-1)^2}{q}=\frac{(m-1)^2}{mf+1}=\frac{(m-1)^2}{m(\lambda m-1)+1}=\frac{(m-1)^2}{\lambda m^2-m+1}<\frac{(m-1)^2}{m^2-2m+1}=1.
$$
Hence the assumption in (b) forces $g_0^2=1/(m+1)$, i.e., $q=(m-1)^2(m+1)$. This implies that $m-1$ and $m+1$ are both powers of $p$. If $m>2$, then $p=2$ since $2=(m+1)-(m-1)$ is divisible by $p$, but this results in a contradiction that $m$ is odd as $m$ divides $q-1$. When $m=2$, however, we get $q=3$ and $|M_{q,m}|=2$, contrary to the assumption that $M_{q,m}$ is a nontrivial difference set. Thus (b) holds.
\end{proof}

\subsection{System on $(\hat{g},\theta)$-level}

Assume that $m$ is even for the rest of the section. In this subsection, we apply DFT to the system on $g$-level, which will lead to equivalent systems. Given an integer $\theta$, we introduce the system of equations
\begin{equation}\label{21}
\begin{cases}
m^2\hat{g}_s\hat{g}_{\frac{m}{2}+s}=\left(\sum\limits_{t=0}^{m-1}\hat{g}_t\right)^2+m^2(m-1),\quad s=0,\dots,\frac{m}{2}-1,\\
m\sum\limits_{t=0}^{m-1}\hat{g}_t\hat{g}_{s+t}=\left(\sum\limits_{t=0}^{m-1}\hat{g}_t\right)^2-m^2,\quad s=0,\dots,\frac{m}{2}-1,\\
\sum\limits_{t=0}^{m-1}(-1)^t\hat{g}_t\hat{g}_{2s-2\theta-t}=\left(\hat{g}_s+\hat{g}_{\frac{m}{2}+s}\right)\sum\limits_{t=0}^{m-1}(-1)^t\hat{g}_t,\quad s=0,\dots,\frac{m}{2}-1,
\end{cases}
\end{equation}
in the unknowns $\hat{g}_0,\hat{g}_1,\dots,\hat{g}_{m-1}$, where subscripts of $\hat{g}$'s are counted modulo $m$, and call it the \emph{system of order $m$ on $(\hat{g},\theta)$-level}.

\begin{notation}
For any $\theta\in\mathbb{Z}$, denote the affine variety consisting of solutions $(\hat{g}_0,\hat{g}_1,\dots,\hat{g}_{m-1})\in\mathbb{C}^m$ to the system of order $m$ on $(\hat{g},\theta)$-level by $\hat{L}_{m,\theta}$.
\end{notation}

The following proposition holds readily.

\begin{proposition}\label{p5}
Let $\theta$ and $\theta'$ be integers.
\begin{itemize}
\item[(a)] If $\theta'\equiv\theta\pmod{m/2}$, then $\hat{L}_{m,\theta'}=\hat{L}_{m,\theta}$.
\item[(b)] If $(\hat{g}_0,\hat{g}_1,\dots,\hat{g}_{m-1})\in\hat{L}_{m,\theta}$, then $(-\hat{g}_0,-\hat{g}_1,\dots,-\hat{g}_{m-1})\in\hat{L}_{m,\theta}$.
\item[(c)] If $(\hat{g}_0,\hat{g}_1,\dots,\hat{g}_{m-1})\in\hat{L}_{m,\theta}$ and $r$ is an integer which is coprime to $m$, then $(\hat{g}_0,\hat{g}_r,\dots,\hat{g}_{(m-1)r})$ with subscripts modulo $m$ lies in $\hat{L}_{m,r\theta}$.
\end{itemize}
\end{proposition}

\begin{theorem}\label{t4}
The map given by
\begin{equation}\label{23}
g_s=\frac{1}{m}\sum\limits_{t=0}^{m-1}\zeta_m^{st}\hat{g}_t,\quad s=0,1,\dots,m-1
\end{equation}
and $h=\zeta_{m/2}^\theta$ is a bijection from $\bigcup_{\theta=0}^{m/2-1}\hat{L}_{m,\theta}$ to $L_m$.
\end{theorem}

\begin{proof}
First of all, we note that (\ref{23}) gives a one to one correspondence between the points $(g_0,g_1,\dots,g_{m-1})$ and $(\hat{g}_0,\hat{g}_1,\dots,\hat{g}_{m-1})$ in $\mathbb{C}^m$.

Suppose $(g_0,g_1,\dots,g_{m-1},h)\in L_m$. Then $h=\zeta_{m/2}^\theta$ for some $\theta\in\{0,1,\dots,m/2-1\}$ since $h^{m/2}=1$, and (\ref{23}) determines a point $(\hat{g}_0,\hat{g}_1,\dots,\hat{g}_{m-1})\in\mathbb{C}^m$. Define a function $g$ on $\mathbb{Z}/m\mathbb{Z}$ by letting $g(s)=g_t$ whenever $s\equiv t\pmod{m}$, $t=0,\dots,m-1$. According to Proposition \ref{p3}(b) and (\ref{23}) we get that $\hat{g}(s)=\hat{g}_t$ whenever $s\equiv t\pmod{m}$, $t=0,\dots,m-1$. For each $s\in\mathbb{Z}$, let
$$
W(s)=\sum\limits_{t=0}^{m-1}(-1)^tg(t)g(s-t).
$$
One readily sees that $W$ is a function on $\mathbb{Z}/m\mathbb{Z}$, and $W(s)=0$ when $s$ is odd. It then follows by (\ref{22}) that $W(1)=\dots=W(m-1)=0$ and
$$
W(0)=g(0)^2+\sum\limits_{t=1}^{m-1}(-1)^tg(t)g(-t)=g(0)^2+m-1.
$$
In light of Proposition \ref{p3} we have
$$
\hat{g}(s)\hat{g}\left(\frac{m}{2}+s\right)=\hat{W}(s)=\sum\limits_{t=0}^{m-1}\zeta_m^{-st}W(t)=W(0)=g(0)^2+m-1,\quad s\in\mathbb{Z}
$$
and
$$
g(0)=\frac{1}{m}\hat{\hat{g}}(0)=\frac{1}{m}\sum\limits_{t=0}^{m-1}\hat{g}(t),
$$
whence
$$
m^2\hat{g}(s)\hat{g}\left(\frac{m}{2}+s\right)=\left(\sum\limits_{t=0}^{m-1}\hat{g}(t)\right)^2+m^2(m-1),\quad s\in\mathbb{Z}
$$
Moreover, direct calculation leads to
$$
m\sum\limits_{t=0}^{m-1}\hat{g}(t)\hat{g}(s+t)=m^2g(0)^2-m^2=\left(\sum\limits_{t=0}^{m-1}\hat{g}(t)\right)^2-m^2,\quad s\not\equiv\frac{m}{2}\pmod{m}
$$
and
$$
\sum\limits_{t=0}^{m-1}(-1)^t\hat{g}(t)\hat{g}(2s-2\theta-t)=\left(\hat{g}(s)+\hat{g}\left(\frac{m}{2}+s\right)\right)\sum\limits_{t=0}^{m-1}(-1)^t\hat{g}(t),\quad s\in\mathbb{Z}.
$$
Hence $(\hat{g}_0,\hat{g}_1,\dots,\hat{g}_{m-1})=(\hat{g}(0),\hat{g}(1),\dots,\hat{g}(m-1))$ satisfies (\ref{21}), i.e.,
$$
(\hat{g}_0,\hat{g}_1,\dots,\hat{g}_{m-1})\in\hat{L}_{m,\theta}.
$$

Conversely, suppose $(\hat{g}_0,\hat{g}_1,\dots,\hat{g}_{m-1})\in\hat{L}_{m,\theta}$ for some $\theta\in\{0,1,\dots,m/2-1\}$ and $(g_0,g_1,\dots,g_{m-1},h)\in\mathbb{C}^{m+1}$ is given by (\ref{23}) and $h=\zeta_{m/2}^\theta$. We aim to show that $(g_0,g_1,\dots,g_{m-1},h)\in L_m$. The equation $h^{m/2}=1$ is clearly satisfied. Since $\sum_{t=0}^{m-1}\hat{g}_t=mg_0$, one derives from (\ref{21}) that
\begin{equation*}
\begin{cases}
\hat{g}_s\hat{g}_{\frac{m}{2}+s}=g_0^2+m-1,\quad s\in\mathbb{Z},\\
\sum\limits_{t=0}^{m-1}\hat{g}_t\hat{g}_{s+t}=mg_0^2-m,\quad s\not\equiv\frac{m}{2}\pmod{m},\\
\sum\limits_{t=0}^{m-1}(-1)^t\hat{g}_t\hat{g}_{2s-2\theta-t}=\left(\hat{g}_s+\hat{g}_{\frac{m}{2}+s}\right)\sum\limits_{t=0}^{m-1}(-1)^t\hat{g}_t,\quad s\in\mathbb{Z},
\end{cases}
\end{equation*}
where subscripts of $\hat{g}$ are counted modulo $m$. If we also count the subscripts of $g$ modulo $m$, then for $s=1,\dots,m/2-1$
\begin{align*}
\sum\limits_{t=0}^{m-1}(-1)^tg_tg_{2s-t}
=&\frac{1}{m^2}\sum\limits_{t=0}^{m-1}(-1)^t\sum\limits_{r=0}^{m-1}\zeta_m^{tr}\hat{g}_r\sum\limits_{r=0}^{m-1}\zeta_m^{(2s-t)j}\hat{g}_j\\
=&\frac{1}{m^2}\sum\limits_{r=0}^{m-1}\sum\limits_{j=0}^{m-1}\zeta_m^{2sj}\hat{g}_r\hat{g}_j\sum\limits_{t=0}^{m-1}\zeta_m^{(m/2+r-j)t}\\
=&\frac{1}{m}\sum\limits_{r=0}^{m-1}\zeta_m^{2s(m/2+r)}\hat{g}_r\hat{g}_{\frac{m}{2}+r}\\
=&\frac{g_0^2+m-1}{m}\sum\limits_{r=0}^{m-1}\zeta_m^{2sr}\\
=&0.
\end{align*}
Thus the first line of (\ref{22}) holds. For $s=1,\dots,m/2$
\begin{align*}
g_sg_{m-s}=&\frac{1}{m^2}\sum\limits_{r=0}^{m-1}\zeta_m^{sr}\hat{g}_r\sum\limits_{t=0}^{m-1}\zeta_m^{(m-s)t}\hat{g}_t\\
=&\frac{1}{m^2}\sum\limits_{r=0}^{m-1}\sum\limits_{t=0}^{m-1}\zeta_m^{s(r-t)}\hat{g}_r\hat{g}_t\\
=&\frac{1}{m^2}\sum\limits_{j=0}^{m-1}\zeta_m^{sj}\sum\limits_{t=0}^{m-1}\hat{g}_{j+t}\hat{g}_t\\
=&\frac{mg_0^2-m}{m^2}\sum\limits_{\substack{j=0\\ j\neq m/2}}^{m-1}\zeta_m^{sj}+\frac{(-1)^s}{m^2}\sum\limits_{t=0}^{m-1}\hat{g}_t\hat{g}_{\frac{m}{2}+t}\\
=&-\frac{mg_0^2-m}{m^2}(-1)^s+\frac{g_0^2+m-1}{m}(-1)^s\\
=&(-1)^s,
\end{align*}
which proves the second line of (\ref{22}). Similarly, one can show that
$$
h^sg_sg_{m+s}=\frac{1}{m^2}\sum\limits_{r=0}^{m/2-1}\zeta_m^{2sr}\left(\hat{g}_r+\hat{g}_{\frac{m}{2}+r}\right)\sum\limits_{t=0}^{m-1}(-1)^t\hat{g}_t
=g_{2s}g_\frac{m}{2}
$$
for $s=1,\dots,m/2-1$, and so the third line of (\ref{22}) is satisfied. This completes our proof.
\end{proof}

Theorem \ref{t4} roughly says that the system of order $m$ on $g$-level decomposes into $m/2$ systems on $(\hat{g},\theta)$-level of the same order. This helps to reduce computation in the next section when $m$ gets large. In fact, we avoid the high degree equation $h^{m/2}=1$ in the the system of order $m$ on $g$-level and instead solve $\tau(m/2)$ systems of order $m$ on $(\hat{g},\theta)$-level, where $\tau$ is the number-of-divisors function. To illustrate it in more detail, we need the following number-theoretic result.

\begin{lemma}\label{l5}
Let $\theta$ be an integer. Then there exists a prime number $r>m$ such that $r\theta\equiv\gcd(\theta,m/2)\pmod{m/2}$.
\end{lemma}

\begin{proof}
Let $\ell=m/2$, $d=\gcd(\theta,\ell)$, $\theta_0=\theta/d$ and $\ell_0=\ell/d$. Then $\gcd(\theta_0,\ell_0)=1$, and so there exists an integer $s$ such that $s\theta_0\equiv1\pmod{\ell_0}$. As $\gcd(s,\ell_0)=1$, by Dirichlet's theorem on arithmetic progressions, there are infinitely many prime numbers congruent to $s$ modulo $\ell_0$. In particular, there is a prime number $r>m$ such that $r\equiv s\pmod{\ell_0}$. It follows that $r\theta_0\equiv1\pmod{\ell_0}$. Consequently, $r\theta_0d\equiv d\pmod{\ell_0d}$, which turns out to be $r\theta\equiv d\pmod{\ell}$, as desired.
\end{proof}

Combining Lemma \ref{l5} with parts (a) and (c) of Proposition \ref{p5} we see that, among the systems of order $m$ on $(\hat{g},\theta)$-level for all integers $\theta$, it suffices to solve for $\theta$ being the divisors of $m/2$. By computation in this $\tau(m/2)$ systems on $(\hat{g},\theta)$-level, we get all the information for the system of order $m$ on $g$-level via Theorem \ref{t4}. This is carried out in the next section up to $m=22$.

\section{Computation results and conjectures}

\subsection{Results for $m\leqslant22$}

One of the main tools for solving systems of polynomial equations is the \emph{Gr\"{o}bner basis} computation. Generally speaking, a Gr\"{o}bner basis is a particular kind of generating set of an ideal in a polynomial ring. (The reader is referred to \cite{cox2007ideals} for the explicit definition and an introduction to this topic.) Once a Gr\"{o}bner basis is computed, it is easy to know many important properties of the ideal and the associated algebraic variety, such as the dimension \cite[\S3 Chapter 9]{cox2007ideals}.

Here we apply a state-of-the-art algorithm of Faug\`{e}re \cite{faugere1999new} called \emph{F4} to compute Gr\"{o}bner bases of the systems order $m$ on $(\hat{g},\theta)$-level with $m=6$, $10$, $12$, $14$, $16$, $18$, $20$ and $22$ performed in \textsc{Maple 14}. For each $m$, we only compute for $\theta$ being a divisor of $m/2$ (when $\theta=m/2$ it is equivalent to put $\theta=0$). It turns out that in these computed cases, $\hat{L}_{m,\theta}$ is a finite set, and thus each $(\hat{g}_0,\hat{g}_1,\dots,\hat{g}_{m-1})$ satisfies
$$
F_{m,\theta}\left(\frac{1}{m}\sum\limits_{t=0}^{m-1}\hat{g}_t\right)=0
$$
for some univariate polynomial\setcounter{footnote}{0}\footnote{One can find them by the command \textbf{Groebner[UnivariatePolynomial]} in\textsc{ Maple 14}.} $F_{m,\theta}(x)$ listed in Tables \ref{tab1}--\ref{tab8} ($F_{m,\theta}(x)=1$ indicates that $\hat{L}_{m,\theta}$ is empty). Now for $m=6$, $10$, $12$, $14$, $16$, $18$, $20$ and $22$, Theorem \ref{t4} in conjunction with Lemma \ref{l5} and parts (a) and (c) of Proposition \ref{p5} implies that each $(g_0,g_1,\dots,g_{m-1},h)$ satisfies $F_m(g_0)=0$, where
$$
F_m(x)=\prod\limits_{\theta\mid n}F_{m,\theta}(x).
$$
We list $F_m(x)$ for these values of $m$ in Table \ref{tab9}\setcounter{footnote}{1}\footnote{To find $F_m(x)$, one may also compute directly in the system of order $m$ on $g$-level. However, the author's PC failed to compute a Gr\"{o}bner basis for the system on $g$-level when $m=22$ as it is too memory-consuming.}.

\vfill
\eject

\begin{table}[htbp]\label{tab1}
\caption{$F_{6,\theta}(x)$}\label{tab1}
\centering
\begin{tabular}{|c|c|c|}
\hline
$\theta$ & $0$ & $1$\\
\hline
$F_{6,\theta}(x)$ & $(x-2)(x+2)$ & $7x^2-1$\\
\hline
\end{tabular}
~\\
~\\
~\\
~\\
\end{table}

\begin{table}[htbp]
\caption{$F_{10,\theta}(x)$}\label{tab2}
\centering
\begin{tabular}{|c|c|c|}
\hline
$\theta$ & $0$ & $1$\\
\hline
$F_{10,\theta}(x)$ & $x(x-4)(x+4)$ & $11x^2-1$\\
\hline
\end{tabular}
~\\
~\\
~\\
~\\
\end{table}

\begin{table}[htbp]
\caption{$F_{12,\theta}(x)$}\label{tab3}
\centering
\begin{tabular}{|c|c|c|c|c|}
\hline
$\theta$ & $0$ & $1$ & $2$ & $3$\\
\hline
$F_{12,\theta}(x)$ & $1$ & $13x^2-1$ & $1$ & $(x-3)(x+3)(x-5)(x+5)(5x-7)(5x+7)$\\
\hline
\end{tabular}
~\\
~\\
~\\
~\\
\end{table}

\begin{table}[htbp]
\caption{$F_{14,\theta}(x)$}\label{tab4}
\centering
\begin{tabular}{|c|c|c|}
\hline
$\theta$ & $0$ & $1$\\
\hline
$F_{14,\theta}(x)$ & $(x-6)(x+6)(4x^2+3)$ & $1$\\
\hline
\end{tabular}
~\\
~\\
~\\
~\\
\end{table}

\begin{table}[htbp]
\caption{$F_{16,\theta}(x)$}\label{tab5}
\centering
\begin{tabular}{|c|c|c|c|c|}
\hline
$\theta$ & $0$ & $1$ & $2$ & $4$\\
\hline
$F_{16,\theta}(x)$ & $(7x-17)(7x+17)$ & $1$ & $17x^2-1$ & $(x-7)(x+7)$\\
\hline
\end{tabular}
~\\
~\\
~\\
~\\
\end{table}

\begin{table}[htbp]
\caption{$F_{18,\theta}(x)$}\label{tab6}
\centering
\begin{tabular}{|c|c|c|c|}
\hline
$\theta$ & $0$ & $1$ & $3$\\
\hline
$F_{18,\theta}(x)$ & $x(x-8)(x+8)$ & $19x^2-1$ & $1$\\
\hline
\end{tabular}
~\\
~\\
~\\
~\\
\end{table}

\begin{table}[htbp]
\caption{$F_{20,\theta}(x)$}\label{tab7}
\centering
\begin{tabular}{|c|c|c|c|c|}
\hline
$\theta$ & $0$ & $1$ & $2$ & $5$\\
\hline
$F_{20,\theta}(x)$ & $1$ & $1$ & $1$ & $(x-7)(x+7)(x-9)(x+9)(9x-31)(9x+31)(13x-67)(13x+67)$\\
\hline
\end{tabular}
~\\
~\\
~\\
~\\
\end{table}

\begin{table}[htbp]
\caption{$F_{22,\theta}(x)$}\label{tab8}
\centering
\begin{tabular}{|c|c|c|}
\hline
$\theta$ & $0$ & $1$\\
\hline
$F_{22,\theta}(x)$ & $(x-10)(x+10)(4x^4-60x^2+243)$ & $23x^2-1$\\
\hline
\end{tabular}
~\\
~\\
~\\
~\\
\end{table}

\begin{table}[htbp]
\caption{$F_m(x)$}\label{tab9}
\centering
\begin{tabular}{|c|c|}
\hline
$m$ & $F_m(x)$\\
\hline
$6$ & $(x-2)(x+2)(7x^2-1)$\\
\hline
$10$ & $x(x-4)(x+4)(11x^2-1)$\\
\hline
$12$ & $(x-3)(x+3)(x-5)(x+5)(5x-7)(5x+7)(13x^2-1)$\\
\hline
$14$ & $(x-6)(x+6)(4x^2+3)$\\
\hline
$16$ & $(x-7)(x+7)(7x-17)(7x+17)(17x^2-1)$\\
\hline
$18$ & $x(x-8)(x+8)(19x^2-1)$\\
\hline
$20$ & $(x-7)(x+7)(x-9)(x+9)(9x-31)(9x+31)(13x-67)(13x+67)$\\
\hline
$22$ & $(x-10)(x+10)(4x^4-60x^2+243)(23x^2-1)$\\
\hline
\end{tabular}
\end{table}

\vfill
\eject

Due to the these computation results we have the following theorem.

\begin{theorem}\label{t11}
If $m\leqslant22$ is an even integer other than $2$, $4$ or $8$, then neither $H_{q,m}$ nor $M_{q,m}$ forms a nontrivial difference set in $\mathbb{F}_q$.
\end{theorem}

\begin{proof}
Let $(g_0,g_1,\dots,g_{m-1},h)$ be any point in $L_m\cap(\mathbb{R}^*\times\mathbb{C}^m)$. Then $F_m(g_0)=0$ with $F_m(x)$ lying in Table \ref{tab9}. Note that $4x^4-60x^2+243=0$ has no solution in $\mathbb{R}$. We infer from Table \ref{tab9} that either $g_0^2\geqslant1$ or $g_0^2=1/(m+1)$. Accordingly, neither $H_{q,m}$ nor $M_{q,m}$ forms a nontrivial difference set in $\mathbb{F}_q$ by Theorem \ref{t10}(b).
\end{proof}

\subsection{Conjectural classification}\label{sec1}

Let us summarize what has been known so far about the existence of nontrivial $m$th-cyclotomic and modified cyclotomic difference sets in $\mathbb{F}_q$. First, the case when $m$ is odd is dealt with in Theorem \ref{t8}: only $M_{16,3}$ arises as a difference set. Second, for even values of $m$ up to $8$, all the nontrivial $m$th-cyclotomic and modified cyclotomic difference sets in $\mathbb{F}_q$ have been determined due to Paley \cite{paley1933orthogonal}, Hall \cite{hall1965characters} and Storer \cite{storer1967cyclotomy}: they are the quadratic case with $q\equiv3\pmod{4}$ and quartic and octic case with $q=p$ as in (\ref{32})--(\ref{35}). For even values of $m$ from $10$ to $22$, $H_{q,m}$ and $M_{q,m}$ do not form nontrivial difference sets any more as shown in Theorem \ref{t11}. Now we pose a conjectural classification.

\begin{conjecture}\label{conj1}
$H_{q,m}$ is a nontrivial difference set in $\mathbb{F}_q$ if and only if one of the following appears:
\begin{itemize}
\item[(a)] $m=2$ and $q\equiv3\pmod{4}$;
\item[(b)] $m=4$ and $q=p=1+4t^2$ for some odd integer $t$;
\item[(c)] $m=8$ and $q=p=1+8u^2=9+64v^2$ for some odd integers $u$ and $v$.
\end{itemize}
$M_{q,m}$ is a nontrivial difference set in $\mathbb{F}_q$ if and only if one of the following appears:
\begin{itemize}
\item[(a')] $m=2$ and $q\equiv3\pmod{4}$;
\item[(b')] $m=3$ and $q=16$;
\item[(c')] $m=4$ and $q=p=9+4t^2$ for some odd integer $t$;
\item[(d')] $m=8$ and $q=p=49+8u^2=441+64v^2$ for some odd integer $u$ and even integer $v$.
\end{itemize}
\end{conjecture}

We have seen that Conjecture \ref{conj1} is true for odd $m$ and even $m\leqslant22$. In fact, our verification for even values of $m$ up to $22$ other than $2$, $4$ or $8$ builds on the computation results that the assumptions of both (a) and (b) in Theorem \ref{t10} hold for these $m$'s. Viewing this, we address the following conjecture, whose latter part obviously implies the former.

\begin{conjecture}\label{conj2}
Suppose that $m$ is even and $m\neq2$, $4$ or $8$.
\begin{itemize}
\item[(a)] Each $(g_0,g_1,\dots,g_{m-1},h)\in L_m\cap(\mathbb{R}^*\times(\mathbb{S}^1)^m)$ satisfies $g_0^2\geqslant1/(m+1)$.
\item[(b)] Each $(g_0,g_1,\dots,g_{m-1},h)\in L_m\cap(\mathbb{R}^*\times(\mathbb{S}^1)^m)$ satisfies either $g_0^2\geqslant1$ or $g_0^2=1/(m+1)$.
\end{itemize}
\end{conjecture}

By the benefit of Theorem \ref{t10}, it provides a possible way to tackle Conjecture \ref{conj1} for higher powers by verifying Conjecture \ref{conj2} for larger $m$: if Conjecture \ref{conj2}(b) is true then Conjecture \ref{conj1} is true; if part (a) of Conjecture \ref{conj2} is true then at least the statement about $H_{q,m}$ in Conjecture \ref{conj1} holds. From the author's viewpoint, it is quite possible that, like the computation results in the previous subsection, each $(g_0,g_1,\dots,g_{m-1},h)\in L_m$ satisfies the inequalities in Conjecture \ref{conj2}. In other words, it would probably suffice to compute the Gr\"{o}bner basis for (\ref{22}) or (\ref{21}) for even $m\geqslant24$.

\subsection{Flag-transitive projective planes}\label{sec5}

A \emph{finite projective plane} of order $n$, where $n\in\mathbb{Z}_{>1}$, is a point-line incidence structure satisfying:
\begin{itemize}
\item[(i)] each line contains exactly $n+1$ points and each point is contained in exactly $n+1$ lines;
\item[(ii)] any two distinct lines intersect in exactly one point and any two distinct points are contained in exactly one line.
\end{itemize}
The incident point-line pairs are called \emph{flags}. A permutation on the point set preserving the lines and flags is called a \emph{collineation} or \emph{automorphism}. If the collineation group of a finite projective plane acts $2$-transitively on the points, then it is said to be \emph{$2$-transitive}. If the collineation group of a finite projective plane acts transitively on the flags, then it is said to be \emph{flag-transitive}. Note that $2$-transitive finite projective planes are always flag-transitive because two distinct points determine a line.

From the definition of difference sets one sees that each $(v,k,1)$-difference set $D$ in an abelian group $F$ gives rise to a finite projective plane once we call elements of $F$ points and $\{D+a\mid a\in F\}$ lines. For such consideration, $(v,k,1)$-difference sets are also called \emph{planar difference sets}. A finite projective plane coordinatized by a finite field is said to be \emph{Desarguesian} since Moufang revealed its equivalence to a certain configurational property named in honor of G. Desargues (see for example \cite{hughes1973projective}). An elegant and celebrated theorem of Wagner \cite{wagner1965finite} asserts that every finite $2$-transitive projective plane is Desarguesian, which actually classifies the $2$-transitive projective planes. Toward a generalization of Wagner's theorem, a conjecture was made that every finite flag-transitive projective plane is Desarguesian. This conjecture has received attention of wide scope and is still open as it is attributed to the existence problem of related planar difference sets by Proposition \ref{p2} below. For more about the history of this longstanding conjecture including Proposition \ref{p2}, see \cite{thas2003finite,thas2008finite}.

\begin{proposition}\label{p2}
If there exists a finite non-Desarguesian flag-transitive projective plane of order $m$ with $v$ points, then $v=m^2+m+1$ is prime and $H_{v,m}$ is a $(v,m+1,1)$-difference set in $\mathbb{F}_v$ with $m>8$.
\end{proposition}

An important concept in the theory of difference sets is the so called \emph{multiplier}. Its idea stems from Hall \cite{hall1947cyclic} when investigating the special case of planar difference sets in cyclic groups, and has been generalized to difference sets in an arbitrary group with a lot of outcomes (see the survey \cite{jungnickel1992difference} for example). Nevertheless, we will focus on the abelian group case for our purposes, where a (\emph{numerical}) \emph{multiplier} of a difference set $D$ in an abelian group $F$ is defined to be an integer $t$ with $\gcd(t,|F|)=1$ such that $tD=D+a$ for some $a\in F$. The following result is due to Chowla and Ryser \cite{chowla1950combinatorial}.

\begin{proposition}\label{p6}
Let $D$ be a $(v,k,\lambda)$-difference set in an abelian group. If $t$ is a prime divisor of $k-\lambda$ with $\gcd(t,v)=1$ and $t>\lambda$, then $t$ is a multiplier of $D$.
\end{proposition}

In \cite[\textsc{Theorem IV}]{lehmer1953residue}, Lehmer proved that the set of multipliers of a nontrivial cyclotomic difference set in $\mathbb{F}_p$ is the difference set itself. We note that this result can be extended to $\mathbb{F}_q$ along the the same lines of proof. Now suppose that $m$ is even and $H_{q,m}$ is a planar difference set. It follows that $f$ is odd by (\ref{12}), and thus the order $f-1$ is even. As Proposition \ref{p6} implies that $2$ is a multiplier of $H_{q,m}$, we then have $2\in H_{q,m}$, and so $\chi(4)=1$ for any multiplicative character $\chi$ of order $m$ on $\mathbb{F}_q$. This allows us to add the equation $h=1$ to (\ref{22}) in order that $H_{q,m}$ is a planar difference set in $\mathbb{F}_q$. Hence the following theorem holds as a consequence of (\ref{12}) and Proposition \ref{p2}.

\begin{theorem}
If $H_{q,m}$ is a planar difference set in $\mathbb{F}_q$, then $q=m^2+m+1$ with $m$ even and the system of order $m$ on $g$-level has a solution $(g_0,g_1,\dots,g_{m-1},h)$ in $\mathbb{R}\times(\mathbb{S}^1)^m$ with $g_0=-1/\sqrt{q}$ and $h=1$. In particular, if there does not exist even integer $m>8$ such that $v=m^2+m+1$ is prime and the system of order $m$ on $g$-level has a solution $(g_0,g_1,\dots,g_{m-1},h)$ in $\mathbb{R}\times(\mathbb{S}^1)^m$ with $g_0=-1/\sqrt{v}$ and $h=1$, then every finite flag-transitive projective plane is Desarguesian.
\end{theorem}

\subsection{Other problems on the systems of equations}

Computation results shows that $L_m$ is a finite set when $m\leqslant22$ is even and $m\neq2$, $4$ or $8$. We thus make another conjecture below. Its affirmative solution for each fixed $m$ will result in a conclusion by Theorem \ref{t1} that there exist at most finitely many $q$'s such that $q\equiv1\pmod{m}$ and either $H_{q,m}$ or $M_{q,m}$ is a difference set in $\mathbb{F}_q$ for this $m$.

\begin{conjecture}
$L_m$ is a finite set for $m\geqslant24$ even.
\end{conjecture}

Studying whether there exists $(g_0,g_1,\dots,g_{m-1},h)\in L_m$ with $g_0=0$ is of interest and importance as well. One of the reasons is that it also has connection with finiteness results by the next theorem.

\begin{theorem}
Suppose that $m$ is even. If $L_m\cap(\{0\}\times(\mathbb{S}^1)^m)$ is empty, then there exist at most finitely many $q$'s such that $q\equiv1\pmod{m}$ and either $H_{q,m}$ or $M_{q,m}$ is a difference set in $\mathbb{F}_q$.
\end{theorem}

\begin{proof}
Let $Q$ be the set of prime powers $q$ such that $q\equiv1\pmod{m}$ and $H_{q,m}$ is a difference set in $\mathbb{F}_q$. For each $q\in Q$, there exists a $(g_0(q),g_1(q),\dots,g_{m-1}(q),h(q))\in L_m\cap(\mathbb{R}\times(\mathbb{S}^1)^m)$ such that $g_0(q)=-1/\sqrt{q}$ by Theorem \ref{t1}. Suppose $Q$ to be infinite. Then there exists an infinite increasing sequence $(q_n)_{n=1}^\infty$ in $Q$. As $(\mathbb{S}^1)^m$ is bounded and closed, it is compact by the Heine-Borel theorem. Hence the sequence $((g_1(q_n),\dots,g_{m-1}(q_n),h(q_n)))_{n=1}^\infty$ in $(\mathbb{S}^1)^m$ has a subsequence
$$
((g_1(q_{n_k}),\dots,g_{m-1}(q_{n_k}),h(q_{n_k})))_{k=1}^\infty
$$
which has a limit point, say $(g_1,\dots,g_{m-1},h)\in(\mathbb{S}^1)^m$. Since
$$
g_0=\lim\limits_{k\rightarrow\infty}g_0(q_{n_k})=\lim\limits_{k\rightarrow\infty}-\frac{1}{\sqrt{q_{n_k}}}=0
$$
and $(g_0(q_{n_k}),g_1(q_{n_k}),\dots,g_{m-1}(q_{n_k}),h(q_{n_k}))$ satisfies (\ref{22}) for every integer $k\geqslant1$, taking the limit $k\rightarrow\infty$ in each polynomial equation of (\ref{22}) we deduce that $(g_0,g_1,\dots,g_{m-1},h)$ satisfies (\ref{22}). This shows that $(g_0,g_1,\dots,g_{m-1},h)\in L_m\cap(\{0\}\times(\mathbb{S}^1)^m)$, contrary to the assumption of the theorem. Consequently, $Q$ is finite. Along similar lines one can prove the finiteness of the set of prime powers $q$ such that $q\equiv1\pmod{m}$ and $M_{q,m}$ is a difference set in $\mathbb{F}_q$. Thus the theorem is true.
\end{proof}

For the values of $m$ in Table \ref{tab9}, $L_m\cap(\{0\}\times\mathbb{C}^m)$ is nonempty only when $m=10$ or $18$. Hence we would like to ask:

\begin{question}
For which even $m$'s is $L_m\cap(\{0\}\times\mathbb{C}^m)$ nonempty?
\end{question}

\bigskip
\noindent\textsc{Acknowledgements.}
The author would like to thank the anonymous referee for the careful reading and helpful suggestions.


\end{document}